\documentclass[11pt]{amsart}
\usepackage{amsmath, amssymb, amscd, amsthm, amsfonts, latexsym}
\usepackage{graphicx}
\usepackage{tikz-cd}
\usepackage{hyperref}
\usepackage{mathtools}
\usepackage{mathrsfs}
\usepackage{enumerate}
\usepackage{upgreek}
\usepackage{mathabx}
\usepackage{setspace}
\usepackage{color}
\usepackage{comment}
\usepackage{tikz}
\usepackage{tikz-cd}

\theoremstyle{plain}
\newtheorem{thm}{Theorem}[section]

\newtheorem{lemm}[thm]{Lemma}

\theoremstyle{definition}
\newtheorem{defn}[thm]{Definition}

\newtheorem{rmk}[thm]{Remark}

\def\Z{{\mathbb Z}}

\def\cd{\protect\operatorname{cd}}
\def\cat{\protect\operatorname{cat}}

\title{On numerical invariants of retraction map}
\author{Nursultan Kuanyshov}

\address{Nursultan Kuanyshov, Suleyman Demirel University, Kaskelen, Kazakhstan}
\email{nursultan.kuanyshov@sdu.edu.kz, kuanyshov.nursultan@gmail.com}

\subjclass[2010]{Primary 50M30; Secondary 55N25, 55M10}

\keywords{Sequential topological complexity, topological complexity, cohomological dimension, Lusternik-Schnirelmann category, retraction map}

\begin{document}
\maketitle
\begin{abstract}
We introduce a notion of retraction between continuous maps of topological spaces and study the behavior of several numerical invariants under such retractions. These include (co)homological dimensions, the Lusternik-Schnirelmann category, the topological complexity, and the sequential topological complexity. We prove that, under the retraction map, the corresponding inequalities between invariants hold. Our results also apply to recent invariants are defined by  Dranishnikov, Jauhari \cite{DJ} and Knudsen, Weinberger \cite{KW}.

\end{abstract}

\section{Introduction}

Over the past five decades, numerical invariants such as (co)homological dimensions, the Lusternik-Schnirelmann category, the topological complexity, and their various generalizations have played a central role in algebraic topology (\cite{CLOT, DK, EFMO, Fa1, Fa2, Fa3, Ku, Pa1, Pa2, Rud}). These invariants quantify the minimal resources needed for certain topological operations, such as covering spaces by contractible subsets, constructing motion planners, or determining (co)homological finiteness properties. They are homotopy invariants and have deep connections to both abstract homotopy theory and applications such as robotics and control systems.

In this paper, we generalize the classical notion of a retraction of topological spaces to the setting of \emph{retraction map} of maps between topological spaces. This broader framework allows us to analyze how numerical invariants behave not just under inclusions and projections, but under more general maps of spaces.

Our main motivation is the following phenomenon (the monotonicity property): In general, if $A\subseteq X$, the inequality $\operatorname{inv}(A) \leq \operatorname{inv}(X)$ for a given numerical invariant $\operatorname{inv}$ may not hold (see a concrete example in Section 4 for Lusternik-Schnirelmann category). However, we prove that if there exists a retraction map from $\mathrm{Id}_X$ to $\mathrm{Id}_A$ where $\mathrm{Id}$ is the identity map between topological spaces X, A respectively, then such inequalities do hold for a broad class of invariants, including the ones mentioned above.

Furthermore, we include recent developments on \emph{numerical invariants of maps}, which extend classical notions from spaces to maps. Notable contributions include the topological complexity of maps introduced by Scott \cite{Sc}, Kuanyshov \cite{Ku} and distributional invariants developed by Dranishnikov, Jauhari \cite{DJ} and  Knudsen, Weinberger \cite{KW}. 

\vspace{1em}
The paper is organized as follows. In Section 2, we give the definition of a retraction map in general settings and recall numerical invariants, namely (co)homological dimensions, the Lusternik-Schnirelmann category, the topological complexity and the sequential topological complexity of space. In Section 3, we prove the cohomological dimension of retraction homomorphisms. In Section 4, we prove the Lusternik-Schnirelmann category of retraction maps. In Section 5, we prove the sequential topological complexity of retraction maps.

In the paper, we use the notation $H^*(\Gamma, A)$ for the cohomology of a group $\Gamma$ with coefficient in $\Gamma$-module $A$. The cohomology groups of a space $X$ with the fundamental group $\Gamma$ we denote as $H^*(X;A)$. Thus, $H^*(\Gamma,A)=H^*(B\Gamma;A)$ where $B\Gamma=K(\Gamma,1)$. The maps are continous functions while spaces are normal topological spaces.  

\section{Retraction Maps and Numerical Invariants}

\subsection{Retraction Map}

Let $f: X \to Y$ and $f': X' \to Y'$ be continuous maps between topological spaces. We say that $f'$ is a \emph{retract of $f$}  if there exist continuous maps
$$
r_X: X \to X', \qquad r_Y: Y \to Y'
$$
such that the following diagram commutes:
$$
\begin{tikzcd}
X \arrow[r,"f"] \arrow[d,"r_X"'] & Y \arrow[d,"r_Y"]\\
X' \arrow[r,"f'"'] & Y'
\end{tikzcd}
\quad \text{with} \quad
f' \circ r_X = r_Y \circ f.
$$

When $f=\mathrm{Id}_X$ and $f'=\mathrm{Id}_A$ for some subspace $A \subseteq X$, this definition reduces to the classical notion of a retraction $r:X\to A$ with $r|_A = \mathrm{id}_A$.

\subsection{Numerical Invariants of Topological Spaces}

We briefly recall the classical numerical invariants considered in this paper.

\begin{defn}
The Lusternik-Schnirelmann category of a topological space X, $\cat(X)$ is the minimal number n such that X admits an open cover by $n+1$ open sets $U_{0},U_{1},..., U_{n}$ and each $U_{i}$ is contractible in X. 
\end{defn}

\begin{defn}[Rudyak, 2009]\label{F TC}
    Let $X$ be a path-connected space.
\begin{enumerate}
\item A {\em sequential motion planner} on a subset $U\subset X^{r}$ is a map $s:U\to PX$ such that $s(x_{0},x_{1},\cdots,x_{r-1})(\frac{j}{r-1})=x_{j}$ for all $j=0,\cdots,r-1.$

\item The {\em sequential topological complexity of a path-connected space $X$}, denoted $TC_{r}(X),$ is the minimal number $k$ such that $X^{r}$ is covered by $k+1$ open sets $U_{0},\cdots,U_{k}$ on which there are sequential motion planners. If no such $k$ exists, we set $TC_{r}(X)=\infty.$
\end{enumerate}
    
\end{defn}

When $r=2$, we recover the famous Farber's topological complexity \cite{Fa1,Fa2,Fa3}.

\begin{defn}
 For a ring $R$, the cohomological dimension of $X$ is
$$\operatorname{cd}_R(X) = \sup \{ n \in \mathbb{N} \mid H^n(X; M) \neq 0 \text{ for some } R\text{-module } M \} $$

\end{defn}

The homological dimension $\operatorname{hd}_R(X)$ is defined analogously using homology. (see more details in \cite{Br}).

In the proof of our main result on the cohomological dimension of retraction homomorphism, we use Shapiro's lemma ~\cite{Br}[Proposition 6.2, page 73]. 

\begin{thm}[``Shapiro's lemma"]{\label{Shapiro}}
    If $i:\Gamma'\to \Gamma$ is a monomorphism and $M$ is an $\Z\Gamma'$-module, then the through homomorphism
     $$H^{*}(\Gamma,Coind_{\Gamma'}^{\Gamma}M)\stackrel{i^*}\to H^*(\Gamma',Coind_{\Gamma'}^{\Gamma}M)\stackrel{\alpha_*}\to H^*(\Gamma',M)$$
     is an isomorphism,
where $Coind_{\Gamma'}^{\Gamma}M=Hom_{\Z\Gamma'}(\Z\Gamma,M)$ and the homomorphism of coefficients $\alpha:Hom_{\Z \Gamma'}(\Z \Gamma,M)\to M$ is defined as $\alpha(f)=f(e)$.
\end{thm}

\section{(Co)homological dimension of retraction homomorphisms}

We recall the {\em cohomological dimension} $\cd(\phi)$ of a group homomorphism $\phi:\Gamma\to\Lambda$ was introduced by Mark Grant ~\cite{Gr} as the maximum of $k$ such that
there is a $\Z\Lambda-$module $M$ with  the nonzero  induced  homomorphism $\phi^*:H^k(\Lambda,M)\to H^k(\Gamma,M)$. When $\phi$ is the identity homomorphism, we recover the classical cohomological dimension of a discrete group $\Gamma$. 

Similarly, one can define the {\em homological dimension} of a group homomorphism $\phi:\Gamma\to\Lambda$ is the maximum of $k$ such that
there is a $\Z\Lambda-$module $M$ with  the nonzero  induced  homomorphism $\phi_*:H_k(\Gamma,M)\to H_k(\Lambda,M)$.

We give the proof for cohomology since the proof for homology is exactly the same with slight modifications. 

Given a group homomorphism $\phi:\Gamma\to \Lambda$. Define a subhomomorphism $\phi':=\phi|_{\Gamma'}$ where $\Gamma'$ is a subgroup of $\Gamma$. 
\begin{lemm}\label{group homomorphism}
    Given a group homomorphism $\phi:\Gamma\to \Lambda$. Any a subhomomorphism $\phi':\Gamma'\to\Lambda'$, $\cd(\phi')\leq \cd(\phi)$.  
\end{lemm}

\begin{proof}
    The proof of the lemma follows from the naturality of the Shapiro lemma, Theorem \ref{Shapiro}. 
\end{proof}

\begin{thm}
 Let $\phi:\Gamma\to\Lambda$ be a homomorphism between the groups $\Gamma,\Lambda$. Let $\phi':\Gamma'\to\Lambda'$ be retraction homomorphism where $\Gamma', \Lambda'$ are subgroups of $\Gamma, \Lambda$ respectively. Then $\cd(\phi)=\cd(\phi')$.    
\end{thm}

\begin{proof}
    By Lemma \ref{group homomorphism}, we get $\cd(\phi')\leq\cd(\phi)$. 

Suppose $\cd(\phi')=k.$ Then, there is a $\Z\Lambda'-$ module $M$ with  the nonzero  induced  homomorphism $\phi^*:H^k(\Lambda',M)\to H^k(\Gamma',M)$. Since $\phi'$ is a retraction map, we get the following commutative diagram

$$
\begin{tikzcd}
\Gamma \arrow[r,"\phi"] \arrow[d,"r_\Gamma"'] & \Lambda \arrow[d,"r_\Lambda"]\\
\Gamma' \arrow[r,"\phi'"'] & \Lambda'
\end{tikzcd}
\quad \text{with} \quad
\phi' \circ r_\Gamma = r_\Lambda \circ \phi.
$$

Since $r_{\Lambda}$ is surjective, $Z\Lambda'$-module M can be considered $Z\Lambda$-module M. Thus, we get the following commutative diagram 
  
$$
\begin{tikzcd}
H^{k}(\Gamma,M)   &   H^{k}(\Lambda,M) \arrow[l,"\phi^{*}"]\\
H^{k}(\Gamma',M) \arrow[u,"r_\Gamma^{*}"] & H^{k}(\Lambda',M) \arrow[l,"\phi'^{*}"] \arrow[u,"r_\Lambda^{*}"]
\end{tikzcd}
$$

Since retraction $r$ and inclusion $i:\Lambda'\to \Lambda$ induce surjective and injective homomorphisms in cohomology with any coefficients, we get the following diagram 

$$
\begin{tikzcd}
H^{k}(\Gamma,M)   &   H^{k}(\Lambda,M) \arrow[l,"\phi^{*}"] \arrow[d,"i_\Lambda^{*}"]\\
H^{k}(\Gamma',M) \arrow[u,"r_\Gamma^{*}"] & H^{k}(\Lambda',M) \arrow[l,"\phi'^{*}"] 
\end{tikzcd}
$$

Suppose $\cd(\phi)\neq k$. Then $\phi(a)=0$ for all elements of $H^{k}(\Lambda, M)$.

Pick an element $a_{1}\in H^{k}(\Lambda, M)$ with $b_{1}=i_{\Lambda}^{*}(a)\neq 0$ since $i_{\Lambda}^{*}$ is an injective homomorphism. ${\phi'}^{*}(b)\neq 0$ by $\cd(\phi')=k$. This is a contradiction since the diagram is a commutative diagram, i.e. $0=\phi^{*}(a_{1})=(r^{*}\circ{\phi'}^{*}\circ i_{\Lambda}^{*})(a)\neq 0$. Therefore, we get $\cd(\phi)=k$. This proves the Theorem.

\end{proof}

\section{Lusternik-Schnirelmann category of retraction maps}

Given  a subspace $X'$ of X i.e. $X'\subset X$, we do not have the monotonicity property $\cat(X')\leq\cat(X).$ Take $X$ to be a disc $D^{2}=\{(x,y)\in R^{2}|x^{2}+y^{2}\leq 1\}$, and $X'$ to be a boundary of disc $D^{2}$, $\{(x,y)\in R^{2}|x^{2}+y^{2}=1\}$. It is easy to see $\cat(X')=1$, $\cat(X)=0$.

If we assume that there is a retraction $r:X\to X'$, then we have the monotonicity property. 

\begin{lemm}
   Given a subspace $X'$ of a topological space X with a retraction $r:X\to X'$. Then $\cat(X')\leq\cat(X)$. 
\end{lemm}
\begin{proof}
 The proof follows from the following observation: A retract of contractible space is contractible.   
 \end{proof}

The LS-category of a topological space extended to the map, Fox \cite{CLOT} gave the following definition. 

 \begin{defn}
 The Lusternik-Schnirelmann category of map f, $\cat(f)$, between X and Y topological spaces is the minimal number n such that X admits an open cover by n+1 open sets $U_{0},U_{1},..., U_{n}$ and restriction of f to each $U_{i}$ is null-homotopic. 
\end{defn}

\textbf{Remark 1:} $\cat(X)=\cat(\mathrm{Id})$ where $\mathrm{Id}: X\rightarrow X$ (identity map)

Let $f:X\to Y$ be map and define $f':=f|X'$ where $X'$ subset of a topological space X. 
We recall the definition of a retraction map below:
\begin{defn}
  We call $f:X\to Y$ is a retraction map if we have the following commuative diagram $$r_{Y}\circ f=f'\circ r_{X}$$ where $f':=f|_{X'}$ namely $f':X'\to Y'$, and $r_{X}:X\to X'$ and $r_{Y}:Y\to Y'$ retractions of X and Y respectively.  
\end{defn}

With above retraction map, we can state the generalized monotonicty property for LS-category of maps.

\begin{thm}
    Given $f:X\to Y$ a retraction map with $f':X'\to Y'$. Then $$\cat(f')\leq\cat(f).$$
\end{thm}
\begin{proof}
Suppose $\cat(f)=n$. Then there exists $n+1$ open subsets of X, name $U_{0},U_{1},\cdots,U_{n}$ such that the union of them covers X, i.e $X=\bigcup_{i=0}^{n}U_{i}$ and the restriction of the map $f$ to each $U_{i}$ is a null-homotopic. In another words, there is a homotopy $H_{i}(x,t)$ between $f|_{U_{i}}$ and contant map. 

For each i, we define $V_{i}=X'\cap U_{i}$. It is easy to see $V_{i}$ are open subset of $X'$ and covers $X'$. 

Claim: $f'|_{V_{i}}$ are null-homotopic for all i. 

For each $V_{i}$, we can explicitly define a homopy $F_{i}:V_{i}\times I\to Y$ by $F_{i}(x,t)=r_{Y}\circ H_{i}(x,t)$ where $r_{Y}:Y\to Y'$ is a retraction.

Since $f:X\to Y$ is a retraction map, we verify the following:
$$F_{i}(x,0)=r_{Y}\circ H_{i}(x,0)=r_{Y}\circ f(x)=f'\circ r_{X}(x)=f'(x)$$

$$F_{i}(x,1)=r_{Y}\circ H_{i}(x,1)=r_{Y}\circ y_{0}=r_{Y}(y_{0})$$

Therefore, $f'|_{V_{i}}$ are null-homotopic for all i. This follows $\cat(f')\leq n$. This proves the Theorem. 
    
\end{proof}

\section{Sequential topological complexity of retraction maps}

Let $f:X\to Y$ be a map. Let $X^{r}$ and $Y^{r}$ be the Cartisian product of $r$ copies of $X$ and $Y$ respectively, i.e $X^{r}:=X\times\cdots\times X$ and $Y^{r}:=Y\times\cdots\times Y.$ Let us denote $f^{r}:=f\times\cdots\times f:X^{r}\to Y^{r}$ and elements of $X^{r}$ and $Y^{r}$ are vectors $\bar{x}=(x_{0},\cdots,x_{r-1})$ and $\bar{y}=(y_{0},\cdots,y_{r-1})$ respectively. Let $PY$ be a based-point path space, i.e. $\{\gamma:[0,1]\to Y|\gamma(0)=y_{0}\}$.  

\begin{defn}

\begin{enumerate} 
\item A {\em sequential $f$-motion planner} on a subset $U\subset X^{r}$ is a map $f_{U}:U\to PY$ such that $f_{U}(\bar{x})(\frac{j}{r-1})=f_{U}(x_{0},x_{1},\cdots,x_{r-1})(\frac{j}{r-1})=f(x_{j})$ for all $j=0,\cdots,r-1.$

\item The {\em sequential topological complexity of map f}, denoted $TC_{r}(f),$ is the minimal number $k$ such that $X^{r}$ is covered by $k+1$ open sets $U_{0},\cdots,U_{k}$ on which there are sequential $f$-motion planners. If no such $k$ exists, we set $TC_{r}(f)=\infty.$
\end{enumerate}

\end{defn}
Note that if $r=2,$ we recover Scott's topological complexity for a map. Further, if map $f$ is identity on space $X,$ we get Rudyak's sequential topological complexity of space $X.$ We need the following technical theorem to prove our main theorem.

\begin{thm}\label{sequential TC}
 Let $f:X\to Y$ be a map, and let $U\subset X^{r}$. The following are equivalent:
 \begin{enumerate}
     \item There is a sequential $f$-motion planner $f_{U}:U\to PY$.
     \item The projections from $f^{r}(U)$ to the j+1 factor of $Y^{r}$ are homotopic, where $j=0,..,r-1$. 
     \item  $f^{r}|_{U}$ can be deformed into the diagonal $\Updelta Y$ of $Y^{r}$. 
       
 \end{enumerate}
\end{thm}
\begin{proof} 

\textbf{$(1 \Rightarrow 2)$} Let $pr_{j}:f^{r}(U)\subset Y^{r}\to Y$ be a projection onto the $j^{th}$ factor of $Y^{r}$. It suffices to show that $pr_{j}$ is homotopic to $pr_{j+1}$ for all $j=0,..,r-2.$ By (2), there is a sequential $f$-motion planner on $U$ and let $\bar{x}=(x_{0},\cdots,x_{r-1})\in U.$

We define the homotopy $$H_{j}(\bar{x},t):=f_{U}(\bar{x})\big(\frac{t+j}{r-1}\big).$$ Then 
$H_{j}(\bar{x},0)=f_{U}(\bar{x})(\frac{j}{r-1})=f(x_{j})=pr_{j}(f^{r}(\bar{x}))$
and
$H_{j}(\bar{x},1)=f_{U}(\bar{x})(\frac{j+1}{r-1})=f(x_{j+1})=pr_{j+1}(f^{r}(\bar{x})).$ Thus, all projections from $f^{r}(U)$ to the $j^{th}$ factor of $Y^{r}$ are homotopic for $j=0,\cdots,r-1.$

\textbf{$(2 \Rightarrow 3)$} Since any two projections from $f^{r}(U)$ to the $j^{th}$ factor of $Y^{r}$ are homotopic, we fix a homotopy $H_{j}:f^{r}(U)\times I\to Y^{r}$ from $pr_{j}$ to $pr_{j+1}$ for $j=0,\cdots,r-2.$ For given $\bar{x}\in U,$ this homotopy $H_{i}(f^{r}(\bar{x}),t)$ is a path from $f(x_{j})$ to $f(x_{j+1})$ and denote this path as $\alpha_{\bar{x}}^{j}.$ We define a concatenation of paths as $\gamma_{\bar{x}}^{i}=*^{i}_{j=0}\alpha_{\bar{x}}^{j}.$ Note that $\gamma_{\bar{x}}^{i}$ is a path from $f(x_{0})$ to $f(x_{i}).$ Now we define a homotopy $H:f^{r}(U)\times I\to Y^{r}$ as

$$H(f^{r}(\bar{x}),t)=(H_{0}(f^{r}(\bar{x}),t(1-t)),\gamma_{\bar{x}}^{1}(1-t),\cdots,\gamma_{\bar{x}}^{r-1}(1-t)).$$

Then $H(f^{r}(\bar{x}),0)=(H_{0}(f^{r}(\bar{x}),0),\gamma_{\bar{x}}^{1}(1),\cdots,\gamma_{\bar{x}}^{r-1}(1))=(f(x_{0}),f(x_{1}),\cdots,f(x_{r-1}))=f^{r}(\bar{x})$ and 
$H(f^{r}(\bar{x}),1)=(H_{0}(f^{r}(\bar{x}),0),\gamma_{\bar{x}}^{1}(0),\cdots,\gamma_{\bar{x}}^{r-1}(0))=(f(x_{0}),f(x_{0}),\cdots,f(x_{0}))\in \Updelta(Y).$ This gives a deformation of $f^{r}|_{U}$ into $\Updelta(Y),$ showing (3). 

\textbf{$(3 \Rightarrow 1)$} Let $H:U\times I\to Y^{r}$ be a deformation of $f^{r}|_{U}$ to $\Updelta(Y).$ We define a map $f_{U}:U\to Y^{r}$ 
$$f_{U}(\bar{x})(t)=\begin{cases} pr_{j}\circ H(\bar{x}, 2(r-1)t-2j) & \mbox{if} \, \, \frac{j}{r-1}\leq t \leq \frac{2j+1}{2(r-1)} \\ pr_{j+1}\circ H(\bar{x}, 2j+2-2(r-1)t) & \mbox{if} \, \, \frac{2j+1}{2(r-1)}\leq t \leq \frac{j+1}{(r-1)}\end{cases}$$
where $j=0,\cdots,r-2$ and $pr_{j}$ and $pr_{j+1}$ are projection of $Y^{r}$ to $Y$ with $j^{th}$ and $(j+1)^{th}$ coordinates respectively. Then $f_{U}$ is well-defined and continous since $H(\bar{x},t)\in \Updelta(Y)$ for all $\bar{x}\in U.$ Moreover, 
$$f_{U}(\bar{x})(\frac{j}{r-1})=pr_{j}\circ H(\bar{x}, 2(r-1)\frac{j}{r-1}-2j)=pr_{j}\circ H(\bar{x},0)=pr_{j}(f^{r}(\bar{x}))=f(x_{j}),$$ and 
$$f_{U}(\bar{x})(\frac{j+1}{r-1})=pr_{j+1}\circ H(\bar{x}, 2j+2-2(r-1)\frac{j+1}{r-1})=pr_{j+1}\circ H(\bar{x},0)=pr_{j+1}(f^{r}(\bar{x}))=f(x_{j+1})$$
so that $f_{U}$ is an sequential $f$-motion planner on U.

\end{proof}

\begin{thm}
    Given $f:X\to Y$ a retraction map with $f':X'\to Y'$. Then $$TC_{r}(f')\leq TC_{r}(f).$$
\end{thm}

\begin{proof}
    Let $TC_{r}(f)=n$. There are $n+1$ open subsets $U_{0},U_{1},\cdots,U_{n}$ that their union covers $X^{r}$ with each $U_{i}$ having sequential f-motion planners. By Theorem \ref{sequential TC}, $f^{r}(U_{i})$ deformed into the diagonal $\Updelta Y$ of $Y^{r}$ for each i. Define $V_{i}:=U_{i}\cap (X')^{r}$ for all $i=1,\cdots,n$. It is easy to see from construction that the union of $V_{i}$ covers $(X')^{r}$. 
    
    Claim: $(f')^{r}(V_{i})$ deforms into the diagonal $\Updelta Y'$ of $(Y')^{r}$ for all i. 

    Since $(f')^{r}(V_{i})\subset f^{r}(U_{i})$ and $f^{r}(U_{i})$ deforms to the diagonal $\Updelta Y$ of $Y^{r}$, we can use the definition of retraction map $f$, i.e.$r_{Y}\circ f=f'\circ r_{X}$, to the diagonal $\Updelta Y$ of $Y^{r}$ deform into the diagonal $\Updelta Y'$ of $(Y')^{r}$. Therefore, we get that $(f')^{r}(V_{i})$ deforms into the diagonal $\Updelta Y'$ of $(Y')^{r}$ for each i. 

    By Theorem \ref{sequential TC} again, we see that for each $V_{i}$ there are sequential f'-motion planners. Therefore, $TC_{r}(f')\leq n.$ This completes the proof of the main theorem.

\end{proof}

\begin{rmk}
   The proof works when $r=2$, which recovers Scott's topological complexity of maps, and when f is the identity $\mathrm{Id}$ recovers Rudyak's sequential topological complexity.    
\end{rmk}

\end{document}